\documentclass[12pt]{amsart}

\usepackage{amsmath, amssymb}
\usepackage{graphicx}
\usepackage{color}
\usepackage[utf8]{inputenc}
\usepackage{float}
\usepackage{subfigure} 
\usepackage{tikz}
\usetikzlibrary{arrows}

\newtheorem{theorem}{Theorem}[section]
\newtheorem{lemma}[theorem]{Lemma}

\theoremstyle{definition}

\theoremstyle{remark}
\newtheorem{remark}[theorem]{Remark}

\numberwithin{equation}{section}

\newcommand{\rr}{{\mathbb R}}

\newcommand{\nat}{{\mathbb N}}
\newcommand{\ganz}{{\mathbb Z}}
\newcommand{\complex}{{\mathbb C}}
\newcommand{\Exp}{{\mathbb E}}

\newcommand{\sign}{\operatorname{sign}}

\newcommand{\Res}{\operatorname{Res}}
\newcommand{\Rea}{\operatorname{Re}}
\newcommand{\eqd}{\stackrel{\rm d}{=}}

\allowdisplaybreaks

\begin{document}

\sloppy
\title[Space-time duality for semi-fractional diffusions]{Space-time duality for semi-fractional diffusions} 
\author{Peter Kern}
\address{Peter Kern, Mathematical Institute, Heinrich-Heine-University D\"usseldorf, Universit\"atsstr. 1, D-40225 D\"usseldorf, Germany}
\email{kern\@@{}hhu.de}

\author{Svenja Lage}
\address{Svenja Lage, Mathematical Institute, Heinrich-Heine-University D\"usseldorf, Universit\"atsstr. 1, D-40225 D\"usseldorf, Germany}
\email{Svenja.Lage@uni-duesseldorf.de} 

\date{\today}

\begin{abstract}
Almost sixty years ago Zolotarev proved a duality result which relates an $\alpha$-stable density for $\alpha\in(1,2)$ to the density of a $\frac1{\alpha}$-stable distribution on the positive real line. In recent years Zolotarev duality was the key to show space-time duality for fractional diffusions stating that certain heat-type fractional equations with a fractional derivative of order $\alpha$ in space are equivalent to corresponding time-fractional differential equations of order $\frac1{\alpha}$. We review on this space-time duality and take it as a recipe for a generalization from the stable to the semistable situation.
\end{abstract}

\keywords{Zolotarev duality, fractional diffusion, semi-fractional derivative, semi\-stable L\'evy process, subordinator, hitting-time}
\subjclass[2010]{Primary 35R11, 60E10; Secondary 26A33, 60G18, 60G22, 60G51, 82C31.}

\maketitle

\baselineskip=18pt

\section{Introduction}\label{sec:1}

Let $X=(X_{t})_{t\geq0}$ be a L\'evy process on $\rr$, i.e.\ a stochastically continuous process starting in $X_{0}=0$ with stationary and independent increments. Assuming that the process is strictly self-similar in the statistical sense that
\begin{equation}\label{selfsim}
(X_{ct})_{t\geq0}\eqd (c^{1/\alpha}X_{t})_{t\geq0}\quad\text{ for all }c>0,
\end{equation}
where $\eqd$ denotes equality of all finite-dimensional marginal distributions of the processes, it is necessarily a stable L\'evy process with parameter $\alpha\in(0,2]$. We will exclude the trivial degenerate case, as well as the cases $\alpha=1$ (Cauchy process) and $\alpha=2$ (Brownian motion), since they are often exceptional. The stable L\'evy process is best characterized by its Fourier transform (FT) in terms of the L\'evy-Khintchine formula $\Exp[e^{ikX_{t}}]=\exp(t\psi(k))$
with log-characteristic function
\begin{equation}\label{logchar}
\psi(k)=i\mu k+\int_{\rr\setminus\{0\}}\left(e^{ikx}-1-\frac{ikx}{1+x^{2}}\right)\,d\phi(x)
\end{equation}
for some unique drift parameter $\mu\in\rr$ and a unique L\'evy measure 
\begin{equation}\label{Lmeas}
d\phi(x)=D\left(p\cdot x^{-\alpha-1}1_{\{x>0\}}+q\cdot |x|^{-\alpha-1}1_{\{x<0\}}\right)\,dx,
\end{equation}
where $D>0$ and $p,q\geq0$ with $p+q=1$. It is well-known that the process has smooth densities $x\mapsto p(x,t)$, i.e.\ they are $C^\infty(\rr)$-functions such that a density itself and all its derivatives belong to $C_0(\rr)\cap L^1(\rr)$. According to \cite{SamTaq}, the FT $\widehat{p}(k,t)=\int_{\rr}e^{ikx}p(x,t)\,dx=\exp(t\psi(k))$ can be parametrized as
\begin{equation}\label{densparam}
\widehat{g}(k;\alpha,\beta,\sigma,v):=\widehat{p}(k,1)=\exp\left(iv k-\sigma^{\alpha}|k|^{\alpha}\left(1-i\beta\sign(k)\tan(\alpha\tfrac{\pi}{2})\right)\right),
\end{equation}
where $\beta=p-q\in[-1,1]$ is a skewness parameter, $\sigma=(D\cdot |\cos(\alpha\tfrac{\pi}{2})|)^{{1/\alpha}}>0$ is a scale parameter, and $v=\mu-\int_{\rr\setminus\{0\}}(\frac{x}{1+x^2}-x\,1_{\{\alpha>1\}})d\phi(x)$ is a centering parameter. In particular, the strict self-similarity \eqref{selfsim} holds iff $v=0$. Moreover, since $X$ is a L\'evy process, for suitable functions $f$ the operators $T_tf(x)=\Exp[f(x-X_t)]$, $t\geq0$, determine a $C_0$-semigroup with generator
\begin{equation}\label{generator}
Lf(x)=-\mu f'(x)+\int_{\rr\setminus\{0\}}\left(f(x-y)-f(x)+\frac{y\,f'(x)}{1+y^2}\right)\,d\phi(x)
\end{equation}
and $\widehat{Lf}(k)=\psi(k)\cdot \widehat{f}(k)$ for $k\in\rr$; e.g., see \cite{SamTaq,Sato} for details.

We will further consider L\'evy processes with a discrete scaling property such that \eqref{selfsim} does only hold for some $c>1$ and thus for all integer powers of $c$, but not necessarily for all $c>0$. These processes are called semistable L\'evy processes and are determined by log-periodic perturbations of the tails  of the L\'evy measure, i.e.\ instead of \eqref{Lmeas} we have for all $x>0$
\begin{equation}\label{semiLmeas}
\phi(x,\infty)=x^{-\alpha}\theta_{+}(\log x)\quad\text{ and }\quad \phi(-\infty,-x)=x^{-\alpha}\theta_{-}(\log x),
\end{equation}
where $\theta_{\pm}$ are non-negative, $\log(c^{1/\alpha})$-periodic functions such that $x\mapsto x^{-\alpha}\theta_{\pm}(\log x)$ are non-increasing, which we call admissable. For details on semistable distributions and L\'evy processes we refer to the monographs \cite{MMMHPS,Sato}. Log-periodic disturbances of power law behavior frequently appears in a variety of physical applications \cite{Sor, ZSP} and also in finance \cite{Sorbook}. In recent years the fractal path behavior of semistable L\'evy processes has been investigated, complementing previous classical results for their stable counterparts. It turned out that in terms of fractal dimension (mainly Hausdorff dimension) the range, the graph and multiple points of the sample paths almost surely are not affected by the log-periodic perturbations \cite{KMX,KerWed,LuksXiao,Wed}, even in terms of exact Hausdorff measure \cite{KerWed2}. Nevertheless, semistable L\'evy processes show a different behavior when turning from fractality to fractionality. When speaking about fractionality, we refer to the well-known result that densities of a stable L\'evy process solve a heat-type partial differential equation (pde) with a fractional derivative operator in space, called the fractional diffusion equation. For details on fractional calculus we refer to the monographs \cite{KST,SKM}.

In Section 2 we will review on this fractional pde approach and a remarkable connection to Zolotarev duality. In the special case of a negatively skewed stable L\'evy process with $\alpha\in(1,2)$, the fractional diffusion equation is known to be equivalent to a time-fractional pde with an ordinary first-order derivative in space, which is called space-time duality \cite{BMN,KelMMM}. This perfectly reflects Zolotarev duality for the related stable densities. From a physical point of view this space-time duality has an important impact. Since fractional derivatives are non-local operators, the fractional diffusion equation lacks of a meaningful physical interpretation. As mentioned by Hilfer \cite{Hil}, due to non-locality in space, experimentally a closed system cannot be separated from its outer environment, whereas non-locality in time does not violate physical principles if one accepts long memory effects.

In Section 3 we ask for a corresponding result concerning the more general class of semistable L\'evy processes. Recently, semi-fractional derivatives have been introduced in \cite{KLM} such that densities of semistable L\'evy processes solve corresponding semi-fractional diffusion equations. This new class of fractional derivatives can be seen as a special case of so-called general fractional derivatives as in \cite{Koc,KKdS}. The approach allows to develop a dual equation with a semi-fractional derivative in time in which the log-periodic disturbances cause an additional inhomogeneity and thus showing a significantly different behavior compared to their stable counterpart. Finally, proofs of our main results are given in Section 4.

\section{Fractional Diffusions and Zolotarev Duality}\label{sec:2}

In this section we follow the arguments laid out in \cite{KelMMM,MMMSik} to derive the probabilistic solution to certain fractional diffusion equations by stable densities, and the approach in \cite{KelMMM} to space-time duality in the negatively skewed case. This is best suitable for our desired generalization towards the semistable setting in Section 3. Consider a fractional diffusion equation of the form 
\begin{equation}\label{fracDE}
\frac{\partial}{\partial t}p(x,t)=-v\frac{\partial}{\partial x}p(x,t)+D\left(\frac{1+\beta}{2}\frac{\partial^\alpha}{\partial x^\alpha}p(x,t)+\frac{1-\beta}{2}\frac{\partial^\alpha}{\partial (-x)^\alpha}p(x,t)\right),
\end{equation}
where $D>0$ if $\alpha\in(1,2)$,  $D<0$ if $\alpha\in(0,1)$, $v\in\rr$ is a velocity parameter, and $\beta\in[-1,1]$ is the skewness parameter. Here $\frac{\partial^\alpha}{\partial x^\alpha}f(x)$ and $\frac{\partial^\alpha}{\partial (-x)^\alpha}f(x)$ denote the positive and negative Riemann-Liouville fractional derivatives defined for suitable functions $f$ as the unique functions with FT $(-ik)^\alpha\widehat{f}(k)$, respectively $(ik)^\alpha\widehat{f}(k)$. For integers $\alpha\in\nat$ this FT coincides with $\int_\rr e^{\pm ikx}f^{(\alpha)}(x)\,dx=\widehat{f^{(\alpha)}}(\pm k)$ and thus fractional derivatives generalize integer order derivatives. Turning to the FT on both sides of \eqref{fracDE} yields
\begin{equation}\label{FTfracDE}\begin{split}
\frac{\partial}{\partial t}\widehat{p}(k,t) & =v\,ik\widehat{p}(k,t)+D\left(\frac{1+\beta}{2}(-ik)^\alpha+\frac{1-\beta}{2}(ik)^\alpha\right)\widehat{p}(k,t)\\
& =v\,ik\widehat{p}(k,t)-\sigma^\alpha |k|^\alpha\left(1-i\beta\sign(k)\tan(\alpha\tfrac{\pi}{2})\right)\widehat{p}(k,t),
\end{split}\end{equation}
where the last equality follows after a short calculation (see equations (5.5) and (5.6) in \cite{MMMSik} for details) with the scale parameter $\sigma=(-D\cos(\alpha\frac{\pi}{2}))^{1/\alpha}>0$. With the initial conditions $\widehat{p}(0,t)=1$ for a probability density, and $\widehat{p}(k,0)=1$ corresponding to the point source $p(x,0)=\delta(x)$, using \eqref{densparam} the unique solution to the ode \eqref{FTfracDE} is given by $\widehat{p}(k,t)=\widehat{g}(k;\alpha,\beta,\sigma t^{1/\alpha}, vt)$, showing that the stable densities $p(x,t)=g(x;\alpha,\beta,\sigma t^{1/\alpha}, vt)$ solve \eqref{fracDE}.

We now restrict our considerations to the negatively skewed case $\beta=-1$ with $\alpha\in(1,2)$, $v=0$ and $D=1$. The corresponding fractional diffusion equation
\begin{equation}\label{fracDEns}
\frac{\partial}{\partial t}p(x,t)=\frac{\partial^\alpha}{\partial (-x)^\alpha} p(x,t)
\end{equation}
is solved by the stable densities
\begin{equation}\label{solsfracDE}
p(x,t)=g\left(x;\alpha,-1,\left(\left|\cos\left(\alpha\tfrac{\pi}{2}\right)\right| t\right)^{1/\alpha},0\right).
\end{equation}
Applying the Fourier-Laplace transform (FLT) $\bar{p}(k,s)=\int_{0}^\infty\int_{\rr}e^{-st+ikx}p(x,t)\,dx\,dt$ to both sides of \eqref{fracDEns} yields $s\,\bar{p}(k,s)-1=(ik)^{\alpha}\bar{p}(k,s)$ for the point source fulfilling $\widehat{p}(k,0)=1$ with solution
\begin{equation}\label{FLTsol}
\bar{p}(k,s)=\frac{1}{s-(ik)^{\alpha}}=\frac{1}{s-\psi(k)},
\end{equation}
where $\psi$ is as in \eqref{logchar} for the L\'evy measure $\phi$ concentrated on the negative axis with $\phi(-\infty,-x)=x^{-\alpha}\frac{\alpha-1}{\Gamma(2-\alpha)}$ and $\mu=\int_{-\infty}^{0}(\frac{x}{1+x^{2}}-x)\,d\phi(x)$. Note that $\bar{p}$ has a single pole at $k=-i\,s^{1/\alpha}$. Inverting the FT by the help of Cauchy's residue theorem (details are given in Section 4), for $x>0$ this leads to
\begin{equation}\label{LTsol}
\widetilde{p}(x,s)=\frac{1}{\alpha}\,s^{-1+1/{\alpha}}\exp\left(-x\,s^{1/\alpha}\right)=\frac{1}{\alpha}\,\widetilde{h}(x,s)
\end{equation}
for the Laplace transform (LT) $\widetilde{p}(x,s)=\int_{0}^\infty e^{-st}p(x,t)\,dt$ as shown in \cite{KelMMM}, where $\widetilde{h}$ is the LT of the inverse $\frac{1}{\alpha}$-stable subordinator with $\frac{1}{\alpha}\in(\frac12,1)$ and density
\begin{equation}\label{solsubord}
h(x,t)=\alpha t\,x^{-1-\alpha}g\left(t\,x^{-\alpha};\tfrac{1}{\alpha},1,\left|\cos\left(\tfrac1{\alpha}\tfrac{\pi}{2}\right)\right| ^{\alpha},0\right)
\end{equation}
for $x>0$; see \cite{ctrw} or equation (4.47) in \cite{MMMSik}. Combining \eqref{solsfracDE}, \eqref{LTsol} and \eqref{solsubord} directly leads to Zolotarev's duality result relating negatively skewed $\alpha$-stable densities for $\alpha\in(1,2)$ with positively skewed $\frac{1}{\alpha}$-stable densities:
\begin{theorem}[\cite{Zol}, Theorem 1]\label{2.1}
For $\alpha\in(1,2)$ and stable densities $g$ parametrized as in \eqref{densparam} we have for all $x>0$ and $t>0$
$$g\left(x;\alpha,-1,\left(\left|\cos\left(\alpha\tfrac{\pi}{2}\right)\right| t\right)^{1/\alpha},0\right)=t\,x^{-1-\alpha}g\left(t\,x^{-\alpha};\tfrac{1}{\alpha},1,\left|\cos\left(\tfrac1{\alpha}\tfrac{\pi}{2}\right)\right| ^{\alpha},0\right).$$
\end{theorem}
Note that Zolotarev uses a different parametrization which can be transferred to the above parametrization \eqref{densparam} as described in \cite{BMN}. Zolotarev proved this result in \cite{Zol} by transforming the FT of the $\alpha$-stable density using complex contour integrals; cf. also Theorem 2.3.1 in \cite{Zolbook}. Lukacs \cite[Theorem 3.3]{Luk} gave a different proof using a series representation of stable densities independently obtained by Bergstr\"om \cite{Ber} and Feller \cite{Fel}. In this work of Feller the $\alpha$-stable density is also shown to be a solution to a fractional diffusion equation with a fractional integral operator of negative order $-\alpha$. It is worth to mention that Zolotarev duality also holds for arbitrary values of the skewness parameter $\beta$, but then the following interpretation as a solution of a time-fractional pde fails. Zolotarev's result further holds for $\alpha=2$ which leads to a closed form expression of a positively skewed $\frac12$-stable density, the only closed form expression known besides the Gaussian and the Cauchy density. This density is frequently called L\'evy density due to its appearance in \cite{Levy}, but according to section 3.7 in \cite{Das} it was already observed by Heavyside in 1871. The fractional pde connection for the case $\alpha=2$ can be found in \cite{BMN2}.

Coming back to duality, we now want to show that \eqref{solsubord} is related to a time-fractional pde. Therefore, applying FT for $x>0$ to \eqref{LTsol} yields $\bar{h}(k,s)=\frac{s^{-1+1/\alpha}}{s^{1/\alpha}-ik}$ which leads to the equation
$$s^{1/\alpha}\bar{h}(k,s)-s^{-1+1/\alpha}=ik\,\bar{h}(k,s).$$
Inverting the FT on both sides gives
\begin{equation}\label{LTtimefrac}
s^{1/\alpha}\widetilde{h}(x,s)-s^{-1+1/\alpha}\delta(x)=-\frac{\partial}{\partial x}\widetilde{h}(x,s).
\end{equation}
For suitable functions $f$ and $t\geq0$ denote by $(\frac{\partial}{\partial t})^{\gamma}f(t)$ the Caputo fractional derivative of order $\gamma\in(0,1)$ which is the unique function with LT $s^{\gamma}\widetilde{f}(s)-s^{\gamma-1}f(0)$, whereas the Riemann-Liouville fractional derivative $\frac{\partial^\gamma}{\partial t^\gamma}$ of order $\gamma\in(0,1)$ is the unique function with LT $s^{\gamma}\widetilde{f}(s)$. Then Laplace inversion on both sides of \eqref{LTtimefrac} yields
\begin{equation}\label{timefrac}
\left(\frac{\partial}{\partial t}\right)^{1/\alpha}h(x,t)=-\frac{\partial}{\partial x}h(x,t)
\end{equation}
for $x>0$ and $t>0$. Since $p(x,t)=\alpha^{-1}h(x,t)$ by \eqref{LTsol}, the original $\alpha$-stable density $p$ also solves the time-fractional pde \eqref{timefrac} under point source initial condition $p(x,0)=\delta(x)$ leading directly to space-time duality for fractional diffusions:
\begin{theorem}[\cite{BMN,KelMMM}]\label{2.2}
For $x>0$ and $t>0$ the point source solutions of the fractional diffusion equation \eqref{fracDEns} of order $\alpha\in(1,2)$ and of the time-fractional pde \eqref{timefrac} of order $\frac1{\alpha}\in(\frac12,1)$ are equivalent.
\end{theorem}
The proof in \cite{BMN} directly uses Zolotarev duality, whereas the above arguments from \cite{KelMMM} only use FLT techniques and gives the partial result on Zolotarev duality stated in Theorem \ref{2.1} as a byproduct. In the semistable setup corresponding duality results are not known and the above FLT method is our preferable choice in Section 3.

To illustrate Theorem \ref{2.2} we plotted numerical solutions $p(x,t)$ of the fractional diffusion equation \eqref{fracDEns}
and $h(x,t)$ of the time-fractional pde \eqref{timefrac} for fixed $t_{0}=3.5$ and $\alpha=1.5$ in Figure 1. For the stable density $p(x,t_{0})$ in \eqref{solsfracDE} we use a Fourier inversion technique together with the representation \eqref{densparam}, whereas $h(x,t_{0})$ was approximated from \eqref{timefrac} by a finite difference method \cite{MMMTad} involving Gr\"unwald-Letnikov differences for the time-fractional derivative. Note that in Figure 1 the ratio $h(x,t_{0})/p(x,t_{0})$ decreases from the true value $\alpha=1.5$ at $x=0$ almost linearly to $1.2$ at $x=4$ which is an effect of the rather weak approximation by Gr\"unwald-Letnikov differences for which the error increases with the distance from the origin.
\begin{figure}\label{fig1} 
\includegraphics[scale=.55]{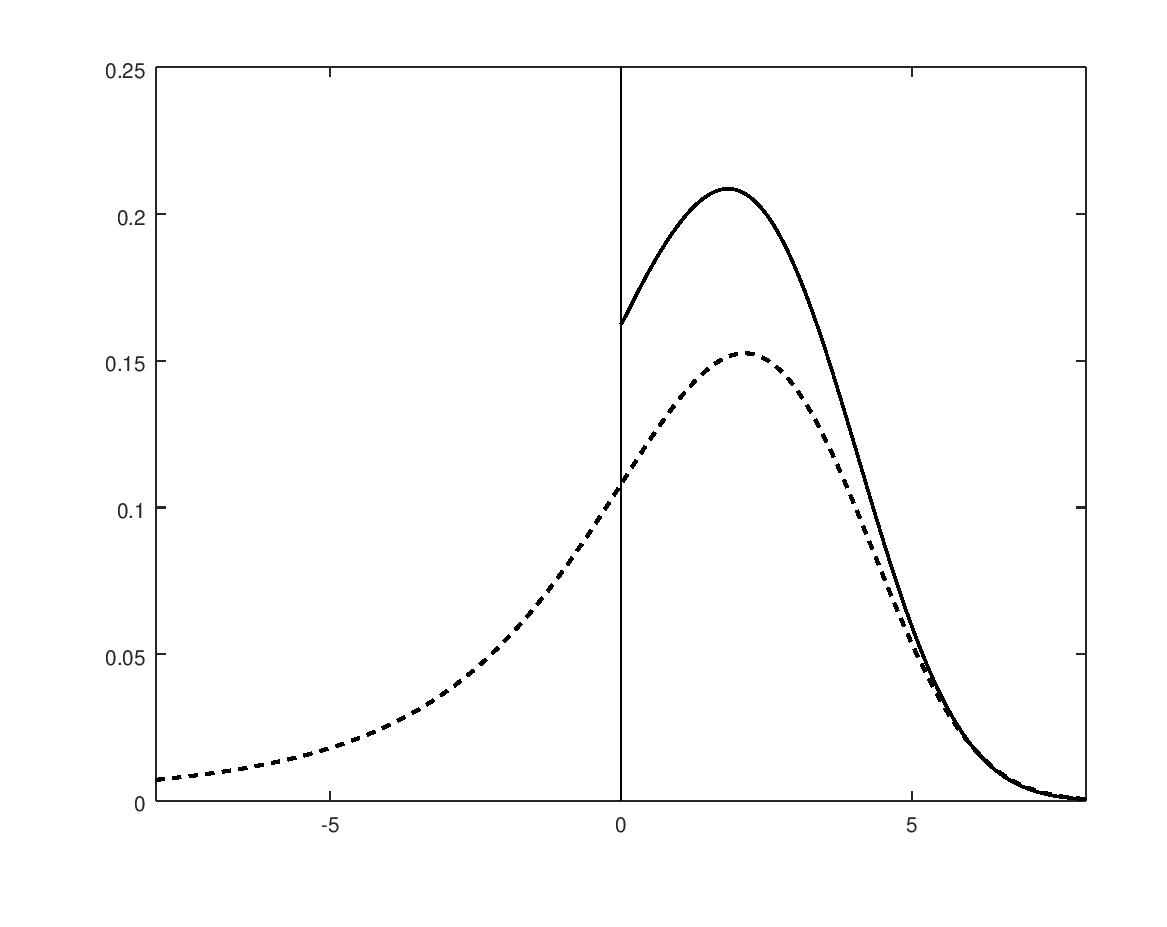}
\caption{Solutions $p(x,t_{0})$ (dashed line) of the fractional diffusion equation \eqref{fracDEns}
and $h(x,t_{0})$ (solid line) of the time-fractional pde \eqref{timefrac} for fixed $t_{0}=3.5$ and $\alpha=1.5$}
\end{figure}
\begin{remark}
The space-time duality in Theorem \ref{2.2} does not cover the full range $\frac1{\alpha}\in(0,1)$ for $\frac1{\alpha}$-stable subordinators. Extending Theorem \ref{2.2} for $\frac1{\alpha}\in(0,\frac12)$ would lead to an equivalent space-fractional pde of order $\alpha>2$ for which in its full generality no meaningful stochastic solution exists. A first result towards this direction is given in \cite{KelMMM2} for $\frac1{\alpha}\in(\frac13,\frac12)$ leading to a probabilistic interpretation of a space-fractional pde of order $\alpha\in(2,3)$ by means of an inverse $\frac1{\alpha}$-stable subordinator. This stochastic solution is much stronger than the higher order approach in \cite{BDOM}.
\end{remark}

\section{Duality for Semi-Fractional Diffusions}\label{sec:3}

We now turn to a negatively skewed semistable distribution for $\alpha\in(1,2)$ with a L\'evy measure $\phi$ as in \eqref{semiLmeas} concentrated on the negative axis
$$\phi(-\infty,-x)=x^{-\alpha}\theta(\log x)\quad,\quad x>0.$$
Here $\theta$ is an admissable function, i.e.\ $\theta$ is a positive, $\log(c^{1/\alpha})$-periodic function for some $c>1$ and $x\mapsto x^{-\alpha}\theta(\log x)$ is non-increasing. We will further assume that $\theta$ is smooth, i.e.\ $\theta$ is continuous and piecewise continuously differentiable, hence representable by a Fourier series
$$\theta(x)=\sum_{n\in\ganz}c_{n}\,e^{in\widetilde{c} x}\quad \text{ with }\quad\widetilde{c}=\frac{2\pi\alpha}{\log c}.$$
In the special case of constant $\theta\equiv c_{0}=\frac{\alpha-1}{\Gamma(2-\alpha)}$ and $\mu=\int_{-\infty}^{0}(\frac{x}{1+x^{2}}-x)\,d\phi(x)$ in \eqref{logchar} this reduces to the stable distribution corresponding to the fractional diffusion equation \eqref{fracDEns}. For the more general semistable distribution  with the same drift parameter $\mu$ the corresponding semi-fractional diffusion equation is given by
\begin{equation}\label{semifracDE}
\frac{\partial}{\partial t}p(x,t)=\frac{\partial^\alpha}{\partial_{c,\theta} (-x)^\alpha} p(x,t).
\end{equation}
Here, for suitable functions $f$ the negative semi-fractional derivative of order $\alpha\in(1,2)$ was recently introduced in \cite{KLM} by its generator form
\begin{equation}\label{semifracder}\begin{split}
\frac{\partial^\alpha}{\partial_{c,\theta} (-x)^\alpha}f(x)=Lf(x) & =\int_{-\infty}^{0}\left(f(x-y)-f(x)+yf'(x)\right)\,d\phi(y)\\
& =\int_{0}^{\infty}\left(f'(x+y)-f'(x)\right)y^{-\alpha}\theta(\log y)\,dy,
\end{split}\end{equation}
where the last equality follows from reflection and integration by parts. As shown in \cite{KLM}, with this definition the negatively skewed semistable densities $x\mapsto p(x,t)$ are a solution to \eqref{semifracDE}. Moreover, it was shown in \cite{KLM} that the corresponding log-characteristic function admits the series representation
\begin{equation}\label{logcharsemi}
\psi(k)=-\sum_{n\in\ganz}c_{n}\,\Gamma(in\widetilde{c}-\alpha+1)(ik)^{\alpha-in\widetilde{c}}
\end{equation}
which for the stable case $\theta\equiv c_{0}=\frac{\alpha-1}{\Gamma(2-\alpha)}=-\frac1{\Gamma(1-\alpha)}$ reduces to $\psi(k)=(ik)^{\alpha}$ and gives back the negative Riemann-Liouville fractional derivative of order $\alpha\in(1,2)$. Applying the FLT on both sides of \eqref{semifracDE} again yields $\bar{p}(k,s)=\frac{1}{s-\psi(k)}$ as in \eqref{FLTsol} for the corresponding semistable densities, but now with $\psi$ from \eqref{logcharsemi}. We will show in Lemma \ref{4.1} that the FLT $\bar{p}$ has again a single pole at some $k=-i\xi(s)$ on the negative imaginary axis which enables us to invert the FT by the help of Cauchy's residue theorem to come to:
\begin{theorem}\label{3.1}
For $\alpha\in(1,2)$ the LT with respect to time of the semistable densities corresponding to the the semi-fractional diffusion equation \eqref{semifracDE} takes the form
\begin{equation}\label{LTsemifrac}
\widetilde{p}(x,s)=\frac1{\alpha}\,\frac{s^{1/\alpha}g(\log s)\exp\left(-x\,s^{1/\alpha}g(\log s)\right)}{s+f(s)}=:\frac1{\alpha}\,\widetilde{h}(x,s),
\end{equation}
where $g$ is a continuously differentiable, $\log(c)$-periodic function and $f$ is some specific function such that $s+f(s)>0$. Moreover, $f$ and $g$ only depend on $c>1$, $\alpha\in(1,2)$ and the admissible function $\theta$.
\end{theorem}
The proof of Theorem \ref{3.1} is given in Section 4. As in Section 2 we now calculate the FT of $\widetilde{h}$ on the right-hand side of \eqref{LTsemifrac} and then apply FLT inversion which also justifies the LT notation $\widetilde{h}(x,s)$ in Theorem \ref{3.1}. Writing $\xi(s)=s^{1/\alpha}g(\log s)$ to simplify notation (it turns out that this is indeed the location of the pole of $\bar{p}(k,s)$ on the negative imaginary axis stated above) and applying FT for $x>0$ to \eqref{LTsemifrac} yields
\begin{align*}
\bar{h}(k,s)& =\frac{\xi(s)}{s+f(s)}\int_0^\infty\exp\left(-x\left(\xi(s)-ik\right)\right)\,dx\\
& =\frac{\xi(s)}{s+f(s)}\,\frac{1}{\xi(s)-ik}=\left(\frac1{s}-\frac1{s}\,\frac{f(s)}{s+f(s)}\right)\frac{\xi(s)}{\xi(s)-ik},
\end{align*}
which leads to the equation
$$\xi(s)\bar{h}(k,s)-s^{-1}\xi(s)-ik\bar{h}(k,s)=-\frac1{s}\,\frac{f(s)}{s+f(s)}\xi(s)=:\frac1{s}\,s^{1/\alpha}\gamma(\log s).$$
Inverting the FT on both sides gives
\begin{equation}\label{LTtimefracsemi}
\xi(s)\widetilde{h}(x,s)-s^{-1}\xi(s)\delta(x)+\frac{\partial}{\partial x}\widetilde{h}(x,s)=\frac1{s}\,s^{1/\alpha}\gamma(\log s)\delta(x).
\end{equation}
We will show in Lemma \ref{4.3} that $\gamma$ is a smooth $\log(c)$-periodic function and thus $\gamma$ and $g$ from Theorem \ref{3.1} both admit a Fourier series representation
\begin{equation}\label{gdef}
g(x)=\sum_{n\in\ganz}d_{n}\,e^{-in\widetilde{d} x}\quad \text{ and }\quad\gamma(x)=\sum_{n\in\ganz}h_{n}\,e^{-in\widetilde{d} x}
\end{equation}
with $\widetilde{d}=\frac{2\pi}{\log c}=\frac{2\pi\frac1{\alpha}}{\log d}
$ for $d=c^{1/\alpha}>1$. Let us define the functions
\begin{equation}\label{taudef}
\tau(x)=\sum_{n\in\ganz}\frac{d_{n}}{\Gamma(in\widetilde{d}-\frac1{\alpha}+1)}\,e^{in\widetilde{d} x}\quad \text{ and }\quad \rho(x)=\sum_{n\in\ganz}\frac{h_{n}}{\Gamma(in\widetilde{d}-\frac1{\alpha}+1)}\,e^{in\widetilde{d} x}
\end{equation}
which clearly are $\log(d^{\alpha})$-periodic functions. Note that formally $\tau(-\log s)$ and $\rho(-\log s)$  are related to $\xi(s)=s^{1/\alpha}g(\log s)$ and $s^{1/\alpha}\gamma(\log s)$ in the same manner than $\theta(-\log(ik))$ is related to $-\psi(k)$ in \eqref{logcharsemi}, simply by multiplying the Fourier coefficients with appropriate values of the gamma function depending on the admissability parameters. We conjecture that $\tau$ and $\rho$ are admissable with respect to the parameters $d>1$ and $\frac1{\alpha}\in(\frac12,1)$. If so, then for suitable functions $f$ and $t\geq0$ we may formally introduce the Riemann-Liouville and the Caputo semi-fractional derivative by LT inversion in analogy to time-fractional derivatives:
\begin{align*}
\frac{\partial^{1/\alpha}}{\partial_{d,\tau}\,t^{1/\alpha}}f(t)=r(t) & \quad\iff\quad\widetilde{r}(s)=\xi(s)\widetilde{f}(s),\\
\left(\frac{\partial}{\partial_{d,\tau}t}\right)^{1/\alpha}f(t)=r(t) & \quad\iff\quad\widetilde{r}(s)=\xi(s)\widetilde{f}(s)-s^{-1}\xi(s)f(0).
\end{align*}
\begin{remark}\label{3.2}
It is worth to mention that this formal introduction of semi-fractional derivatives for functions on the positive real line can be strengthened from a probabilistic perspective. In fact the densities $h(x,t)$ of an inverse $\frac1{\alpha}$-semistable subordinator with a $\log(d^\alpha)$-periodic admissable function $\tau$ in the positive tail of the L\'evy measure solve the semi-fractional pde
$$\left(\frac{\partial}{\partial_{d,\tau}t}\right)^{1/\alpha}h(x,t)=-\frac{\partial}{\partial x}h(x,t)$$
in analogy to \eqref{timefrac} for the densities of an inverse $\frac1{\alpha}$-stable subordinator. This fact is outside the scope of this article and will be published elsewhere.
\end{remark}
Finally, since $\frac1{s}=\int_0^\infty e^{-st}\, dt$ is the LT of the function $1_{(0,\infty)}(t)$, we may now rewrite \eqref{LTtimefracsemi} as
\begin{equation}\label{semitimefrac}
\left(\frac{\partial}{\partial_{d,\tau}t}\right)^{1/\alpha}h(x,t)+\frac{\partial}{\partial x}h(x,t)=\delta(x)\,\frac{\partial^{1/\alpha}}{\partial_{d,\rho}\,t^{1/\alpha}}1_{(0,\infty)}(t).
\end{equation}
Similar to \eqref{semifracder}, for suitable functions $f$ the semi-fractional Caputo derivative of order $\frac1{\alpha}\in(0,1)$ (here we have $\frac1{\alpha}\in(\frac12,1)$) with respect to $d>1$ and the admissable function $\rho$ is given in \cite{KLM} by
\begin{equation}\label{Csemifracder}
\left(\frac{\partial}{\partial_{d,\rho} t}\right)^{1/\alpha}f(t)=\int_{0}^{\infty}f'(t-s) s^{-1/\alpha}\rho(\log s)\,ds
\end{equation}
and the corresponding Riemann-Liouville derivative is obtained by interchanging differentiation and integration on the right-hand side of \eqref{Csemifracder}. Hence, on the right-hand side of \eqref{semitimefrac} we get
\begin{align*}
\frac{\partial^{1/\alpha}}{\partial_{d,\rho}\,t^{1/\alpha}}1_{(0,\infty)}(t) & =\frac{d}{dt}\int_{0}^{\infty}1_{(0,\infty)}(t-s) \,s^{-1/\alpha}\rho(\log s)\,ds\\
& = \frac{d}{dt}\int_{0}^{t}s^{-1/\alpha}\rho(\log s)\,ds=t^{-1/\alpha}\rho(\log t)
\end{align*}
which yields 
\begin{equation}\label{semitimefracinhom}
\left(\frac{\partial}{\partial_{d,\tau}t}\right)^{1/\alpha}h(x,t)+\frac{\partial}{\partial x}h(x,t)=\delta(x)\,t^{-1/\alpha}\rho(\log t).
\end{equation}
Thus we have shown space-time duality for semi-fractional diffusions:
\begin{theorem}\label{3.3}
Assume that $\tau$ and $\rho$ in \eqref{taudef} are admissable functions with respect to the parameters $d=c^{1/\alpha}>1$ and $\frac1{\alpha}\in(\frac12,1)$. Then for $x>0$ and $t>0$ the point source solutions of the semi-fractional diffusion equation \eqref{semifracDE} of order $\alpha\in(1,2)$ in space and of the semi-fractional pde \eqref{semitimefracinhom} of order $\frac1{\alpha}\in(\frac12,1)$ in time are equivalent, i.e.\ $p(x,t)=\alpha^{-1}h(x,t)$ for all $x>0$ and $t>0$.
\end{theorem}
Note that with $f$ and $g$ also $\tau$ and $\rho$ do only depend on $c>1$, $\alpha\in(1,2)$ and the admissable function $\theta$ of the underlying semistable distribution.

\section{Proofs for Section \ref{sec:3}}\label{sec:4}

For simplicity, we write $\omega_n=-c_n\Gamma(in\widetilde{c}-\alpha+1)$ for the coefficients in \eqref{logcharsemi}. Extending $\psi$ for $z\in\complex$ shows that
\begin{equation}\label{logcharsemiC}
\psi(z)=\sum_{n\in\ganz}\omega_{n}(iz)^{\alpha-in\widetilde{c}}=(iz)^\alpha\sum_{n\in\ganz}\omega_{n}e^{-in\widetilde{c}\log(iz)}
\end{equation}
is an analytic function in the lower half plane, where the series in \eqref{logcharsemiC} is absolutely convergent by Theorem 3.1 in \cite{KLM}, and $\psi$ admits the representation
\begin{equation}\label{logcharsemiCint}
\psi(z)=\int_{-\infty}^0\left(e^{izx}-1-izx\right)\,d\phi(x).
\end{equation}
Moreover, since $\omega_{-n}=\overline{\omega_n}$ for $n\in\ganz$, the function 
\begin{equation}\label{logcharsemi-iC}
\psi(-ik)=k^\alpha\sum_{n\in\ganz}\omega_{n}e^{-in\widetilde{c}\log(k)}=:k^\alpha m(\log k)
\end{equation}
for $k>0$ is a real function such that $m$ is $\log(c^{1/\alpha})$-periodic.
\begin{lemma}\label{4.1}
For any $s>0$ there is a unique $z=z(s)$ in the lower half plane such that $s=\psi(z(s))$. Moreover, $z(s)=-i\,\xi(s)$ with $\xi(s)>0$ lies on the negative imaginary axis.
\end{lemma}
\begin{proof}
From \eqref{logcharsemiCint} it can be deduced that for $z$ in the lower half plane $\psi(z)\in\rr$ iff $z=-ik$ with $k>0$. If we consider the real mapping $s(k)=\psi(-ik)$ for $k>0$ then by \eqref{logcharsemiCint}
$$s'(k)=\int_{-\infty}^0x\left(e^{kx}-1\right)\,d\phi(x)>0$$
and thus $k\mapsto s(k)$ is a continuously differentiable and strictly increasing function with $\lim_{k\downarrow0}s(k)=0$ and $\lim_{k\to\infty}s(k)=\infty$. Hence, for $s>0$ there is a unique $\xi(s)>0$ with $s=\psi(-i\,\xi(s))$.
\end{proof}
\begin{lemma}\label{4.2}
The function $\xi$ from Lemma \ref{4.1} is continuously differentiable and for $s>0$ we have $\xi(s)=s^{1/\alpha}g(\log s)$ for some $\log(c)$-periodic function $g$.
\end{lemma}
\begin{proof}
Since $\xi$ is the inverse of the function $k\mapsto s(k)=\psi(-ik)$ appearing in the proof of Lemma \ref{4.1}, it is itself continuously differentiable and strictly increasing. By \eqref{logcharsemi-iC} we get
\begin{align*}
\psi\left(-i\,c^{1/\alpha}\xi(s)\right) & =c\,\xi(s)^\alpha m\left(\log(c^{1/\alpha})+\log\xi(s)\right)\\
& =c\,\xi(s)^\alpha m\left(\log\xi(s)\right)=c\,\psi(-i\,\xi(s))\\
& =cs=\psi(-i\,\xi(cs))
\end{align*}
and thus we have $c^{1/\alpha}\xi(s)=\xi(cs)$. Defining $g(x)=e^{-x/\alpha}\xi(e^x)$ we get
$$g(x+\log c)=e^{-x/\alpha}c^{-1/\alpha}\xi(c\,e^x)=e^{-x/\alpha}\xi(e^x)=g(x).$$
\end{proof}
\begin{proof}[of Theorem \ref{3.1}]
Using equation (4.8.18) in \cite{MorFes}, an inversion of the FT of $\bar{p}(k,s)=(s-\psi(k))^{-1}$ for fixed $s>0$ gives
\begin{equation}\label{FTinverse}
\widetilde{p}(x,s)=\frac1{2\pi}\lim_{T\to\infty}\int_{-T-i\xi_0}^{T-i\xi_0}\frac{e^{-ikx}}{s-\psi(k)}\,dk,
\end{equation}
where we choose $\xi_0\in(0,\xi(s))$. For large $T>0$ consider the cut semicircle $C_T+L_T$ in the lower half plane as in the picture.
\begin{center}
\begin{tikzpicture}
\draw [<-,line width=0.5mm,>=stealth] (-1.95,0) arc (195:346:2cm);
\draw (-2,0.5) arc (180:360:2cm);
\draw [<-,line width=0.5mm,>=stealth] [->] (-1.95,0) -- (1.95,0);
\node at (2,0.7) {$T$};
\node at (-2,0.7) {$-T$};
\node at (1.8,-1.1) {$C_T$};
\node at (0.8,0.25) {$L_T$};
\node at (-0.4,0.15) {\small $-i\,\xi_0$};
\node at (-0.6,-0.8) {\small $-i\,\xi(s)$};
\node at (0,-0.8) {$\bullet$};
\draw [->] (0,-2) -- (0,1);
\draw [->] (-2.5,0.5) -- (2.5,0.5);
\end{tikzpicture}
\end{center}
Letting $k=T\,e^{-i\varphi}$ we get
$$\left|\int_{C_T}\frac{e^{-ikx}}{s-\psi(k)}\,dk\right|\leq\int_0^\pi\frac{T\exp(-Tx\sin\varphi)}{|s-\psi(T\,e^{-i\varphi})|}\,d\varphi\to0$$
as $T\to\infty$ by dominated convergence, since we can easily derive $\Rea\psi(T\,e^{-i\varphi})\to\infty$ for $\varphi\in(0,\pi)$. By Lemma \ref{4.1} and Cauchy's residue theorem we get from \eqref{FTinverse} with the function $s(k)$ from the proof of Lemma \ref{4.1}
\begin{align*}
\widetilde{p}(x,s) & =-i\,\Res(-i\,\xi(s))=\frac{i\,e^{-x\xi(s)}}{\psi'(-i\,\xi(s))}=\frac{e^{-x\xi(s)}}{s'(\xi(s))}\\
& =\frac{e^{-x\xi(s)}}{\xi(s)^{\alpha-1}\left(\alpha\,m(\log \xi(s))+m'(\log \xi(s))\right)}\\
& =\frac1{\alpha}\,\frac{\xi(s)e^{-x\xi(s)}}{\psi(-i\,\xi(s))+\frac1{\alpha}\,\xi(s)^\alpha m'(\log\xi(s))}=\frac1{\alpha}\,\frac{\xi(s)e^{-x\xi(s)}}{s+f(s)},
\end{align*}
where $f(s)=\frac1{\alpha}\,\xi(s)^\alpha m'(\log\xi(s))$. Hence we have shown \eqref{LTsemifrac} and the denominator is strictly positive, since $s+f(s)=\alpha^{-1}\xi(s)\,s'(\xi(s))>0$.
Note that due to the above approach $f$ and $g$ do only depend on the parameters $c$, $\alpha$ and $\theta$ of the semistable distribution.
\end{proof}
\begin{lemma}\label{4.3}
Let $f(s)=\frac1{\alpha}\,\xi(s)^\alpha m'(\log\xi(s))$ as above. Then we can write $\frac{-f(s)}{s+f(s)}\,\xi(s)=s^{1/\alpha}\gamma(\log s)$ for some $\log(c)$-periodic and smooth function $\gamma$.
\end{lemma}
\begin{proof}
Write
$$\frac{-f(s)}{s+f(s)}\,\xi(s)=\frac{-g(\log s)^\alpha m'(\log\xi(s))}{\alpha+g(\log s)^\alpha m'(\log\xi(s))}\,s^{1/\alpha} g(\log s)=s^{1/\alpha}\gamma(\log s).$$
Since $g$ is $\log(c)$-periodic, $m$ is $\log(c^{1/\alpha})$-periodic and $\xi(cs)=c^{1/\alpha}\xi(s)$, the assertion follows easily.
\end{proof}
\begin{remark}
Note that in the stable case we have $\psi(k)=(ik)^\alpha$ and thus $m\equiv1$ in \eqref{logcharsemi-iC} and $g\equiv1$ in Lemma \ref{4.2} are constant. Thus $f\equiv0$ in the above proof of Theorem \ref{3.1} and \eqref{LTsemifrac} coincides with \eqref{LTsol}.
\end{remark}

\bibliographystyle{plain}

\end{document}